\documentclass{amsart}
\usepackage{amssymb,mathtools}
\usepackage[british]{babel}
\usepackage{enumitem}
\usepackage[latin1]{inputenc}
\usepackage[linguistics]{forest}
\usepackage{qtree}
\usepackage{csquotes}
\MakeOuterQuote{"}

\usepackage{url}
\usepackage{xcolor}

\newtheorem{theorem}{Theorem}
\numberwithin{theorem}{section}
\newtheorem{corollary}[theorem]{Corollary}
\newtheorem{lemma}[theorem]{Lemma}
\newtheorem{proposition}[theorem]{Proposition}
\theoremstyle{definition}
\newtheorem{definition}[theorem]{Definition}
\newtheorem{remark}[theorem]{Remark}

\newtheorem*{claim}{Claim}

\newcommand*{\NN}{\mathbb{N}}
\newcommand*{\ZZ}{\mathbb{Z}}
\newcommand*{\QQ}{\mathbb{Q}}

\newcommand*{\ran}{\operatorname{rng}}

\newcommand*{\conc}{\mathbin{\raisebox{0.9ex}{$\smallfrown$}}}

\newcommand*{\Fraisse}{\text{Fra\"{i}ss\'{e}}}

\newcommand{\rank}{\operatorname{rank}}

\newcommand*{\RCA}{{\normalfont {\bf RCA}}}
\newcommand*{\ACA}{{\normalfont {\bf ACA}}}

\newcommand*{\ATR}{{\normalfont {\bf ATR}}}
\newcommand*{\PiCA}{{\normalfont {\bf\Pi^1_1\text{\bf-CA}}}}

\newcommand*{\id}{\operatorname{id}}

\newcommand{\otp}{\operatorname{otp}}

\numberwithin{equation}{section}

\makeatletter
\newcommand{\dotminus}{\mathbin{\text{\@dotminus}}}
\newcommand{\@dotminus}{%
  \ooalign{\hidewidth\raise1ex\hbox{.}\hidewidth\cr$\m@th-$\cr}%
}
\makeatother

\title{Weak well orders and Fra\"iss\'e's conjecture}
\author{Anton Freund and Davide Manca}
\address{University of W\"urzburg, Institute of Mathematics, Emil-Fischer-Stra{\ss}e~40, 97074 W\"urzburg, Germany}
\email{anton.freund@uni-wuerzburg.de{\normalfont, }davide.manca@uni-wuerzburg.de}

\thanks{Funded by the Deutsche Forschungsgemeinschaft (DFG, German Research Foundation) -- Project number 460597863.}

\begin{document}

\begin{abstract}
 The notion of well order admits an alternative definition in terms of embeddings between initial segments. We use the framework of reverse~mathematics to investigate the logical strength of this definition and its connection with Fra\"iss\'e's conjecture, which has been proved by Laver. We also fill a small gap in Shore's proof that Fra\"iss\'e's conjecture implies arithmetic transfinite recursion over~${\bf RCA}_0$, by giving a new proof of $\Sigma^0_2$-induction.
\end{abstract}


\maketitle

\section{Introduction}

The study of well orders is of great importance to proof theory and offers a point of contact between the distinct approaches of ordinal analysis and reverse mathematics. The latter provides a well established framework to compare the axiomatic strength of theorems from various areas. A central idea is to prove equivalences between theorems and axioms over a weak base theory, such as the system~$\RCA_0$ of recursive comprehension. We refer to~\cite{friedman-rm,simpson09} for further background. With respect to that framework, the subsystem $\ATR_0$ appears to be the natural environment for the study of countable {ordinals}. As asserted by S.~Simpson, ``[$\ATR_0$] is the weakest set of axioms which permits the development of a decent theory of countable well orders''~\cite{simpson09}. In particular, $\ATR_0$ is equivalent to the statement that any two countable well orders can be compared \cite{friedman_hirst,simpson09}, in any of the following two ways. For linear orders $X$ and $Y$, an embedding is an order preserving map~$f:X\rightarrow Y$. If such a map exists, we write $X\le_w Y$. On the other hand, a strong embedding from $X$ into $Y$ is an isomorphism between $X$ and an initial segment of $Y$, i.\,e.~a set $I\subseteq Y$ such that we get $x\in I$ whenever we have $x<_Y y$ for some~$y\in I$. Thus, one obtains two quasi orderings over the class of linear orders, and assuming $\ATR_0$ their restrictions to well orders coincide. In view of previous work on these two notions of embeddability (see \cite{friedman_hirst, hirst94}), it seems natural to investigate whether initial segments can be employed to obtain a fruitful characterization of the notion of well order itself. The following is an obvious candidate for such a characterization.

\begin{definition}\label{def:wwo}
A countable linear order~$X$ is called a weak well order if no initial segment~$I\subseteq X$ can be embedded into a proper initial segment~$I_0\subsetneq I$.
\end{definition}

The restriction to countable orders is dictated by the framework of reverse mathematics. However, it does also have a more substantial motivation: Fra\"iss\'e's conjecture, which plays a central role in Proposition~\ref{prop:wwo-to-wo} below, is not available for general uncountable orders. In the following, we assume that all orders are countable. We now note that the property above is entailed by the usual definition of well order.

\begin{lemma}[$\RCA_0$]\label{lem:wo_to_wwo}
If $X$ is a well order, then it is a weak well order.
\end{lemma}
\begin{proof} Aiming to prove the contrapositive, suppose that for some linear order $X$ there exist initial segments $I_0\subsetneq I\subseteq X$ and an embedding $f:I\rightarrow I_0$. Take any $x\in I\setminus I_0$, which is non-empty by assumption: then $x>f(x)>f(f(x))>...$, so that $X$ is not well founded.
\end{proof}

For the converse implication, we rely on the aforementioned Fra\"iss\'e conjecture, proved by R.~Laver~\cite{laver71}. It asserts that any infinite sequence of countable (or more generally $\sigma$-scattered) linear orders~$L_0,L_1,\ldots$ admits $i<j$ with $L_i\leq_w L_j$.

\begin{proposition}[$\RCA_0$]\label{prop:wwo-to-wo}
Fra\"iss\'e's conjecture entails that any weak well order is a well order.
\end{proposition}
\begin{proof}
Aiming for the contrapositive, let $(x_i)_{i<\omega}$ be an infinite descending sequence inside an ill founded linear order $X$. Define $L_i$ as the initial segment of $X$ consisting of all elements smaller than $x_i$. By Fra\"iss\'e's conjecture, there must be $i<j$ such that $I=L_i$ embeds into $I_0=L_j$. Due to $x_i>x_j$, we indeed have $L_j\subsetneq L_i$.
\end{proof}
To see the relation with the two notions of embeddability, observe that if a linear order $X$ embeds into a well order $Y$, regardless of weakly or strongly, then $X$ is also a well order. On the other hand, if $Y$ is just a weak well order, we can still conclude that $X$ is a weak well order when the embedding is strong. However, it is not immediate to reach this conclusion when the embedding is weak. Indeed, we will see that $\mathsf{ATR_0}$ is equivalent to the principle that $X$ is a weak well order whenever we have $X\leq_w Y$ for some weak well order~$Y$ (combine Lemma~\ref{lem:wwo-weak-embed} with Theorem~\ref{thm:atr-wwo}). The notion of well order described in Definition \ref{def:wwo} is weaker than the usual one in another sense as well: Corollary~\ref{cor:omega-omega-rca} provides an example of a linear order that, in a weak enough theory, can be proved to be a weak well order but not a well order.

The exact strength of Fra\"iss\'e's conjecture is an important open problem in reverse mathematics. In~\cite{Shore}, R.~Shore proves that the restriction of the conjecture~to well orders is equivalent to $\ATR_0$, over $\RCA_0+\Sigma^0_2\text{-induction}$. He then argues that the base theory for the latter result can be lowered to just $\RCA_0$. However, the final step that eliminates $\Sigma^0_2\text{-induction}$ uses that~$\omega^\omega$ is well founded, which~$\RCA_0$ cannot prove. A~new proof that Fra\"isse's conjecture implies $\Sigma^0_2$-induction over~$\RCA_0$ will be given in the present paper (see in particular Section~\ref{sect:Fraisse-Sigma2}). Concerning the upper bound, A.~Montalb\'an~\cite{Montalban_fraisse} has shown that Fra\"iss\'e's conjecture is provable in the axiom system~$\PiCA_0$.

It is natural to ask whether the full strength of Fra\"iss\'e's conjecture is needed for Proposition~\ref{prop:wwo-to-wo}, or whether the implication there is wildly inefficient. As it turns out, neither is the case. Based on the following, we will be able to conclude that arithmetic transfinite recursion holds when any weak well order is well founded.

\begin{proposition}[$\RCA_0$]\label{prop:backward_implication}
If every weak well order is a well order,
then the restriction of \Fraisse's conjecture to indecomposable well orders holds.
\end{proposition}
\noindent Before we give the proof, let us recall that a linear order $X$ is indecomposable if, whenever $X=A+B$ holds for non-empty linear orders $A$ and $B$, we have that $X$ embeds into $A$ or that $X$ embeds into $B$. We say that $X$ is indecomposable to the left if it {always} embeds into $A$, and that $X$ is indecomposable to the right if it {always} embeds into $B$. In the special case where $X$ is a well order, Lemma \ref{lem:wo_to_wwo} implies that it {can only be} indecomposable to the right.
\begin{proof}
    Let $(X_i)_{i\in\omega}$ be an infinite sequence of indecomposable well orders: our aim is to find indices $i<j$ such that $X_i$ embeds into $X_j$. We may assume that no $X_i$ is empty. By $\omega^*$ we denote the order on~$\mathbb N$ with order relation ${\leq^*}=\{(m,n)\,|\,m\geq n\}$. Consider the linear order $\sum_{i\in\omega^*}X_i$: it is ill founded, as any family of points $x_i\in X_i$ gives rise to a descending sequence. Given the assumption from the proposition, we can conclude that it is no weak well order. Hence we get an embedding~$f$ from an initial segment $L$ into a shorter initial segment~$L_0$. We find an index $j$ and a non-empty initial segment $I\subseteq X_j$ such that $L=\sum_{i<^*j}X_i+I$. First we prove the thesis under the additional assumption that $L_0=\sum_{i<^*j+1}X_i$, and then we show that this does not violate the generality. 
    Under the additional assumption, $X_{j+1}\subseteq L$ must be embedded into $\sum_{i<^*j+2}X_i+J$ for some initial segment $J\subsetneq X_{j+1}$. In~fact, it embeds into $\sum_{i<^*j+2}X_i$: if some final segment of $X_{j+1}$ did embed into $J$, then so would all of $X_{j+1}$, against Lemma~\ref{lem:wo_to_wwo}. Let $i>j+2$ be the smallest index such that for some $x\in X_{j+1}$, we have $f(x)\in X_i$. Then, a final segment of $X_{j+1}$ embeds into $X_{i}$, and hence so does $X_{j+1}$.
    Now, if $L_0$ is included in $\sum_{i<^*j+1}X_i$, we can simply extend it. The only other possibility is that $L_0=\sum_{i<^*j}X_i+I_0$ holds for some $I_0\subsetneq I\subseteq X_j$. We get that the range of $f\upharpoonright I$ is not contained in~$I_0$, since otherwise the well order $X_j$ would violate Lemma \ref{lem:wo_to_wwo}. Hence, there is a non-empty initial segment $I'\subseteq I$ such that $\ran (f\upharpoonright I')\subseteq L'_0=\sum_{i<^*j+1}X_i$. This means that we can replace $L_0$ with $L'_0$ and $L$ with $L'=\sum_{i<^*j}X_i+I'$, to reduce to the special case that we have already treated.
\end{proof}

 By \cite[Corollary~2.16]{Shore}, the {conclusion} of Proposition~\ref{prop:backward_implication} implies \Fraisse's conjecture for {arbitrary} well orders and thus $\ATR_0$ over $\ACA_0$. Conversely, in Section~\ref{sect:cnf} we adapt Montalb\'an's~\cite{Montalban_fraisse} analysis of Fra\"iss\'e's conjecture via signed trees in order to show that $\ATR_0$ proves that any weak well order is a well order. In Section~\ref{sect:prov-unprov-cases}, we show that the latter implies $\ACA_0$ and is therefore equivalent to arithmetic trans\-finite recursion {over $\RCA_0$} (see Theorem~\ref{thm:atr-wwo}). In the same section, we will also see that, in sharp contrast, $\RCA_0$ suffices to prove that any weak well order that is closed under (a syntactic version of) ordinal exponentiation must already be a well order. Hence the principle that weak well orders are well orders is strong in general but weak in an important class of cases. We will argue (see Remark~\ref{rmk:natural}) that this dichotomy gives some new insight into the idea of `natural' descriptions of linear orders and proof-theoretic ordinals.

 \subsection*{Acknowledgements} {We are very grateful to Richard Shore for information and support with respect to his original proof and its connection with our Section~\ref{sect:Fraisse-Sigma2}.}

\section{Cantor normal form for weak well orders}\label{sect:cnf}

{In the present section, we show that $\ATR_0$ proves that every weak well order is a well order. To to so, we adapt an argument from Montalb\'an's analysis of Fra\"iss\'e's conjecture, in which the notion of Hausdorff rank plays an important role.}

A linear order $X$ is called scattered if $\QQ$ does not embed into it. Since every countable linear order embeds into $\QQ$, whenever $X$ is non-scattered we can consider a pair of embeddings $(f,g)$ such that $f:X\rightarrow \QQ$ and $g:\QQ\rightarrow X$. In general, if we can find such a pair of embeddings between two linear orders, we say that they are equimorphic. 

\begin{lemma}[$\RCA_0$]\label{lem:wwo_scatt}
   Every weak well order is scattered.
\end{lemma}
\begin{proof}
    We prove the contrapositive. Let $L$ be a non-scattered linear order and consider an equimorphism $(f,g)$ between $L$ and $\QQ$. Take $x\in L$ with $x=g(q_0)$ for some $q_0\in\QQ$. Due to the fact that $\QQ$ is indecomposable to the left, we can consider an embedding $h:\QQ\rightarrow \{q\in \QQ\,|\,q < q_0\}$. Then $g\circ h\circ f$ is an embedding of $L$ into the initial segment $\{y\in L\,|\,y<x\}$, so $L$ is not a weak well order.
\end{proof}

As stated in the introduction, all linear orders in the following are assumed to be countable. The main result of this section is that, in $\ATR_0$, any scattered weak well order $W$ has a Cantor normal form: that is to say, there is a well order $\alpha$ and a non-increasing sequence {$\langle\sigma(0),\ldots,\sigma(n-1)\rangle\in \alpha^{<\omega}$} such that $W$ is isomorphic to the order~$\omega^{\sigma(0)}+\ldots+\omega^{\sigma(n-1)}$. This can be {explained} as follows: given a linear order~$X$, we define~$\omega(X)$ as the order with underlying set
\begin{equation*}
\omega(X)=\{\langle x_0,\ldots,x_{n-1}\rangle\,|\,x_i\in X\text{ and }x_0\geq_X\ldots\geq_X x_{n-1}\}
\end{equation*}
and lexicographic comparisons. 
To make things more precise, we write $l(\sigma)$ for the length and $\sigma_i$ for the entries of a sequence $\sigma=\langle\sigma_0,\ldots,\sigma_{l(\sigma)-1}\rangle$. Then $\sigma\leq_{\omega(X)}\tau$ holds precisely if either we have $l(\sigma)\leq l(\tau)$ and $\sigma_i=\tau_i$ for all $i<l(\sigma)$ or there is some $j<\min\{l(\sigma),l(\tau)\}$ with $\sigma_j<_X\tau_j$ and $\sigma_i=\tau_i$ for $i<j$. Accordingly, the Cantor normal form of $W$ can be defined as an element $\sigma\in \omega(\alpha)$ such that $W$ is isomorphic to the initial segment $\{x\in\omega(\alpha)\,|\, x<_{\omega(\alpha)}\sigma\}$, denoted by $\omega[\sigma]$ for short. In the special case where $\sigma=\langle \beta\rangle$, we write the same initial segment as $\omega^\beta$ instead. {For $\sigma=\langle\sigma(0),\ldots,\sigma(n-1)\rangle$, this makes $\omega[\sigma]$ isomorphic to $\omega^{\sigma(0)}+\ldots+\omega^{\sigma(n-1)}$.}

{The promised result on Cantor normal forms entails} that every scattered weak well order is well founded. In light of Lemma \ref{lem:wwo_scatt}, this is sufficient to prove that every weak well order is a well order. Hirst has shown in \cite{hirst94} that $\ATR_0$ is equivalent to the fact that every well order admits a Cantor normal form, but unfortunately his proof does not seem to work for weak well orders. As mentioned before, our approach instead follows previous work by Montalb\'an: the proof of our Theorem~\ref{thm:forward_implication} essentially adapts the one of \cite[Lemma 3.4]{Montalban_equivalence}. Below, we give an informal explanation of how the two arguments are related, but that explanation is not necessary to follow the main argument, save for Remark \ref{rmk:interval}.

In \cite{Montalban_equivalence}, Montalb\'an uses the notion of Hausdorff rank, discussed below, to show that any scattered linear order can be decomposed into the sum of hereditarily indecomposable linear orders. Moreover, assuming a statement equivalent to \Fraisse's conjecture, that sum is finite. Hereditarily indecomposable linear orders can be represented as well founded trees with labels from the set $\{+,-\}$ on each node. The order associated to such a tree $T$ is called the linearization of $T$. The linearization is indecomposable to the left if the label on the root of $T$ is $``-"$, indecomposable to the right if the same label is $``+"$, in the sense explained below Proposition \ref{prop:backward_implication}. In the former case, the linearization is not a weak well order, since it embeds in any of its initial segments. Moreover, any subtree of $T$ represents an interval of the linearization, i.\,e.~a suborder $A$ such that if $x,z\in A$ and  $x<y<z$ then $y\in A$. In general, the following relation holds between a weak well order and its intervals:

\begin{remark}\label{rmk:interval} A linear order $X$ is a weak well order if and only if the same is true for every interval $A\subseteq X$. In fact, suppose that for some interval $A$ we have an embedding $f:J\rightarrow J_0$ of initial segments $J_0\subsetneq J\subseteq A$. Write $X$ as $I+J+I'$ and consider $\id_I$, the identity map on $I$. Then $I+J$ embeds into $I+J'$ via $\id_I+f$.
\end{remark}

The intervals associated to subtrees of $T$ are themselves hereditarily indecomposable linear orders. In light of the previous remark, if $T$ has at least one node with label $``-"$ its linearization is not a weak well order. Conversely, if all the nodes of $T$ have label $``+"$, the linearization of $T$ is a suborder of the Kleene-Brouwer order on a well founded tree related to $T$, and hence $\ACA_0$ proves that it is a well order. It follows by Lemma \ref{lem:wo_to_wwo} that a hereditarily indecomposable linear order is a weak well order if and only if it is the linearization of a tree $T$ with label $``+"$ on every node. Moreover, the linearization is isomorphic to $\omega^{\rank (T)}$ if a rank function on $T$ exists, and in particular when one assumes $\ATR_0$. Intuitively, this is the reason why in the case of weak well orders we are able to obtain a finite decomposition using only $\ATR_0$ instead of $\Fraisse$'s conjecture. Below, we prove this in detail.

First, we recall some known facts about scattered linear orders. Consider a linear order $L$ and a well order $\alpha$. We can assume that the field of $L$ consists of natural numbers. We define by simultaneous transfinite recursion an equivalence relation $\sim_\beta$ on $L$ and a subset $L_\beta\subseteq L$ for all $\beta<\alpha$. Let $\sim_0$ be the identity. We declare that $a\sim_{\beta+1}b$ holds if there exist finitely many points $c_0\le_L...\le_Lc_{n}$ in $L_\beta$ such that either $a\leq_L b$ and $a\sim_\beta c_0\wedge b\sim_\beta c_{n}$ or $b\le_L a$ and $b\sim_\beta c_0\wedge a\sim_\beta c_{n}$, and also for all $x$ between $a$ and $b$ there exists a unique $i\le n$ with $x\sim_\beta c_i$. If $\lambda$ is a limit ordinal, let $a\sim_\lambda b$ hold if we have $a\sim_\gamma b$ for some successor ordinal $\gamma<\lambda$. Moreover, for all $\beta<\alpha$ define $L_\beta=\{x\in L\,|\,x\le_\NN y \text{ for all }y\in L \text{ with } x\sim_\beta y\}$. Note that $L_0=L$ and for all $\beta$, if $x,y$ are distinct elements of $L_\beta$ then $x\not\sim_\beta y$. The set $L_\beta$ is called the $\beta$-th Hausdorff derivative of $L$. If $a\sim_\beta b$, we say that $a$ and $b$ are $\beta$-neighbours. The $\beta$-neighbourhood of $a$ is defined as $N^\beta(a)=\{b\in L\,|\,b\sim_\beta a\}$. It is easy to see that $\beta$-neighbourhoods are intervals. Moreover, if $b\sim_{\beta+1}a$, we have that $N^\beta(b)\subseteq N^{\beta+1}(a)$. Therefore, $N^{\beta+1}(a)$ may be written as $\sum_{b\in I}N^\beta(b)$, where $I=N^{\beta+1}(a)\cap L_\beta$. In $\ATR_0$ we can define the sequence $(\sim_\beta)_{\beta<\alpha}$, for any countable well order $\alpha$.

\begin{theorem}[Clote]\label{thm:clote}
Assume $\ATR_0$ and consider a scattered linear order $L$. Then there exists a countable well order $\alpha$, an $a\in L$ and $\beta<\alpha$ such that $L_\beta=\{a\}$.
\end{theorem}

For a proof we refer to Lemmas~13 and~14 of~\cite{clote_89}. Given $\beta$ and $a$ as in the theorem, we get that $L=N^\beta(a)$. In fact, if $b\in L$ and $c$ has minimal code among the $x\in L$ with $x\sim_\beta b$, it follows that $c\in L_\beta$ and hence $c=a$. But then we have~$b\in N^\beta(a)$. The minimal ordinal $\beta$ such that $L_\beta=\{a\}$ holds for an appropriate $a\in L$ is called the Hausdorff rank of $L$. In the following we adapt the proof of \cite[Lemma~3.4]{Montalban_equivalence} to our needs.

\begin{theorem}[$\ATR_0$]\label{thm:forward_implication}
    Every weak well order~$L$ admits a Cantor normal form, i.\,e.~there is a well order~$\alpha$ and a $\sigma\in\omega(\alpha)$ such that $L$ is isomorphic to $\omega[\sigma]$.
\end{theorem}
\begin{proof}
    If $L$ is a weak well order, it is scattered by Lemma \ref{lem:wwo_scatt}. Therefore, we can use $\ATR_0$ to define the sequence $(\sim_\beta)_{\beta<\alpha}$ for an $\alpha$ that satisfies Theorem~\ref{thm:clote}. Using that sequence as a parameter, we define $L_\beta$ for all $\beta<\alpha$. We aim to define, for all $\beta<\alpha$ and all $a\in L_\beta$, a sequence $\sigma\in\omega(\beta+1)$  and an isomorphism between $N^\beta(a)$ and the initial segment $\omega[\sigma]$. Since Theorem \ref{thm:clote} proves that there exist $\beta<\alpha$ and $a\in L$ with $L=N^\beta(a)$, the sequence $\sigma$ we find in that instance lists the exponents for the Cantor normal form of $L$.
    
    We proceed by arithmetical transfinite recursion, again with $(\sim_\beta)_{\beta<\alpha}$ as a parameter. Take a successor ordinal $\beta+1$ and assume we have defined all the desired sequences and isomorphisms up to level~$\beta$. Consider the $b\in L_\beta\cap N^{\beta+1}(a)$: all of those $b$ are separated from $a$, and hence from each other, by at most finitely many points of $L_\beta$. Hence, they are enumerated by the indices in some $M\subseteq \ZZ$, and we can write $N^{\beta+1}(a)=\sum_{i\in M}N^\beta(b_i)$. Inductively, for each $i\in M$ we have a sequence $\sigma_i\in\omega(\beta+1)$ and an isomorphism between $N^\beta(b_i)$ and $\omega[\sigma_i]$. If $M$ is finite, write it as $\{0,\ldots,n\}$. In that case, we consider the {concatenation} ${\tau=\sigma_0\conc...\conc \sigma_n}$. Let $\sigma$ be the sequence that results from $\tau$ when one deletes any entry that is smaller than a later one. In view of basic ordinal arithmetic, we computably obtain an ~isomorphism between $\sum_{i\in M}\omega[\sigma_i]$ and $\omega[\sigma]$. This relates to the fact that the notation system represents finite sums of additively indecomposable well orders, where every order of the form $\omega^\gamma+\omega^\delta$ with $\gamma<\delta$ is isomorphic to~$\omega^\delta$.
    
     Now suppose $M$ is infinite. We claim that we have $M=\omega$, modulo a change of indices. In fact, suppose that an initial segment $M'$ of $M$ was an infinite descending sequence: in that case, the induction hypothesis would yield an initial segment of $N^{\beta+1}(a)$ isomorphic to $\sum_{i\in \omega^*}\omega^{X(i)}$, where $X$ is an infinite sequence of elements of $\beta+1$ obtained by juxtaposing to the left the finite sequences $\sigma_j$ for $j\in M'$. We observe that there is an index $J$ and an increasing map $h:[J,+\infty[\rightarrow[J+1,+\infty[$ with $X(i)\le_\alpha X(h(i))$ for all $i\ge J$. In fact, if no such $J$ and $h$ did exist, for arbitrarily large $i$ we would find that $X(i)$ is greater than $X(j)$ for all except finitely many~$j$. But this would imply that $\beta+1$ is not well founded. We now see that $\sum_{i\le^*J}\omega^{X(i)}$ embeds into $\sum_{i\le^*J+1}\omega^{X(i)}$. Since isomorphisms preserve weak well orders and intervals, we have that $N^{\beta+1}(a)$ contains an interval which is not a weak well order. But $N^{\beta+1}(a)$ is an interval of $L$, which is a weak well order, so this contradicts Remark \ref{rmk:interval}. This proves our claim that $M=\omega$.
     
     We are now in the case where $N^{\beta+1}(a)$ is isomorphic to $\sum_{i\in\omega}\omega^{Y(i)}$ for an infinite sequence $Y=\sigma_0\conc \sigma_1\conc \sigma_2\dots$ with elements in $\beta+1$. We distinguish two cases. Suppose that there exists an index $j$ such that, for all $k>j$, we have $\beta \not \in \ran(\sigma_k)$, and write $Y=\sigma_0\conc\ldots\conc \sigma_j\conc Y'$. If there existed a $\gamma<\beta$ such that $\gamma\ge \sup Y'$, we would have that for all $i$, all the points in $\omega ^{Y'(i)}$ are $\gamma$-neighbours. In that case, the $\gamma$-th derivative of $\sum_{i\in\omega}\omega ^{Y'(i)}\cong\sum_{k>j}\omega[\sigma_k]$ would be $\omega$, and hence all the points in $\sum_{k>j}\omega[\sigma_k]$  would be $\beta$-neighbours. This contradicts the hypothesis that $M$ is infinite. Then it must be $\sup Y'=\beta$: in that case, we get that $N^{\beta+1}(a)$ is isomorphic to $\sum_{i\le j}\omega[\sigma_i]+\omega^\beta$, so that we can argue as in the case where $M$ is finite. In the other case, $\beta$ occurs infinitely often in $Y$. In that case, we get that $N^{\beta+1}(a)$ is isomorphic to $\omega ^{\beta+1}$. This concludes the discussion of the successor case.
     
     Now consider a limit ordinal $\lambda$ and $a \in L_\lambda$. Fix an increasing sequence of successor ordinals $\beta(i)<\lambda$ with $\lambda=\sup_{i\in\omega} \beta(i)$. Then $N^\lambda(a)=\bigcup_{i<\omega}N^{\beta(i)}(a)$, and each term of the union is isomorphic to $\omega[\sigma_i]$ for an appropriate $\sigma_i\in \omega({\beta(i)+1})$.
    We claim that there exists an $I$ such that $N^{\beta(I)}(a)$ is an initial segment of $N^\lambda (a)$. Suppose not: then we find a strictly increasing subsequence of indices $i_n$ such that $N^{\beta(i_{n+1})}(a)$ extends $N^{\beta(i_n)}(a)$ to the left. In other words, we find intervals $C_n$ such that $C_n+N^{\beta(i_n)}(a)$ is an initial segment of $N^{\beta(i_{n+1})}(a)$. Hence, $C_n$ is isomorphic to an initial segment of $\omega[\sigma_{i_{n+1}}]\subseteq\omega^{\beta(i_{n+1})+1}$. On the other hand, that initial segment must have order type at least $\omega^{\beta(i_n)}$, for otherwise the points in $C_n$ would be \mbox{$\beta{(i_n)}$-}neighbours of those in $N^{\beta(i_n)}(a)$. Moreover, since all the sequences involved are strictly increasing, we can find a strictly increasing map $h:\NN\rightarrow \NN$ verifying $\beta(i_{n+1})+1\le\beta(i_{h(n)})$ for all $n$. We then get ${\otp(C_n)\le \omega^{\beta(i_{n+1})+1}\le \otp (C_{h(n)})}$, so that $C_n$ embeds into $C_{h(n)}$. Thus, the {interval} $\sum_{n\in\omega^*}C_n$ embeds into $\sum_{n\le^*1}C_n$, which contradicts Remark \ref{rmk:interval}.
    
    {Still in the limit case, we now know that $N^{\beta(i)}{(a)}$ is an initial segment of $N^\lambda(a)$ for sufficiently large~$i$. For large $i<j$, the isomorphisms $N^{\beta(i)}(a)\cong\omega[\sigma_i]\subseteq\omega[\sigma_j]$ and $N^{\beta(j)}(a)\cong\omega[\sigma_j]$ must thus agree on $N^{\beta(i)}(a)$, since embeddings between initial segments of well orders are necessarily unique. We can thus glue them to get an isomorphism between $N^\lambda(a)$ and an initial segment of~$\omega^\lambda$, which has the desired form $\omega[\sigma]$ for a suitable~$\sigma\in\omega(\lambda+1)$.}
    \end{proof}

    As explained above, we can conclude the following:

    \begin{corollary}[$\ATR_0$]\label{cor:wwo-to-wo-over-atr}
    Every weak well order is a well order.
    \end{corollary}

   Conversely, one can of course infer Theorem~\ref{thm:forward_implication} from the given corollary and the aforementioned result by Hirst~\cite{hirst94}. At the same time, we find it interesting that our proof via Hausdorff ranks does directly yield Cantor normal forms.

\section{Provable and unprovable cases of weak well foundedness}\label{sect:prov-unprov-cases}

  In this section, we prove the following result and draw several consequences. {The definition of the transformation $X\mapsto\omega(X)$ was recalled in the previous section.}

\begin{theorem}[$\RCA_0$]\label{thm:ord_exp}
A linear order~$X$ is a well order precisely if $\omega(X)$ is a weak well order.
\end{theorem}
\begin{proof}
    First assume that $X$ is no well order. We fix a sequence $x_0>x_1>\ldots$ in~$X$. To show that $\omega(X)$ is no weak well order, we embed it into the proper~initial segment below~$\langle x_0\rangle$. For~$\omega^*$ as in the proof of Proposition~\ref{prop:backward_implication}, we have an embedding
    \begin{align*}
        \omega(\omega^*)&\to\{\sigma\in\omega(X)\,|\,\sigma<_{\omega(X)}\langle x_0\rangle\},\\
        \langle i(0),\ldots,i(n-1)\rangle&\mapsto\langle x_{1+i(0)},\ldots,x_{1+i(n-1)}\rangle.
    \end{align*}
    We claim that $\omega(X)$ embeds into~$\omega(\omega^*)$. Given that the orders considered in reverse mathematics are countable, it suffices to show that $Y:=\omega(\omega^*)\backslash\{\langle\rangle\}$ is an (effectively) dense linear order without endpoints, i.\,e.~isomorphic to~$\mathbb Q$ {(cf.~Lemma~\ref{lem:wwo_scatt})}. To see that an arbitrary $\sigma\in Y$ is no endpoint, we note
    \begin{equation*}
        \langle\sigma_0+1\rangle<_Y\sigma<_Y\langle\sigma_0,\ldots,\sigma_{l(\sigma)-1},\sigma_{l(\sigma)-1}\rangle.
    \end{equation*}
    Now consider an inequality~$\sigma<_Y\tau$. If we have $\sigma_i=\tau_i$ for $i<l(\sigma)<l(\tau)$, we get
    \begin{equation*}
        \sigma<_Y\langle\sigma_0,\ldots,\sigma_{l(\sigma)-1},\tau_{l(\sigma)}+1\rangle<_Y\tau.
    \end{equation*}
    In the remaining case, we have a $j<\min\{l(\sigma),l(\tau)\}$ with $\sigma_i=\tau_i$ for all~$i<j$ as well as $\sigma_j<^*\tau_j$ (which means $\sigma_j>\tau_j$ in~$\mathbb N$). Here we obtain
    \begin{equation*}
        \sigma<_Y\langle\sigma_0,\ldots,\sigma_{l(\sigma)-1},\sigma_{l(\sigma)-1}\rangle<_Y\tau.
    \end{equation*}
    To prove the other direction of our theorem, we now assume that~$X$ is a well order. {In the previous section, we have used $\omega[\sigma]$ as notation for~$\{\rho\in\omega(X)\,|\,\rho<\sigma\}$. To simplify notation, we now agree that $\omega[\sigma]$ can also be denoted by~$\sigma$.} For $\sigma,\tau\in\omega(X)$, we put $\sigma+\tau=\langle\sigma_0,\ldots,\sigma_{i-1},\tau_0,\ldots,\tau_{l(\tau)-1}\rangle$ where $i<l(\sigma)$ is minimal with $\sigma_i<\tau_0$ and $i=l(\sigma)$ if no such index exists. Let us also define $\omega^x\cdot n$ with $x\in X$ and $n\in\mathbb N$ as the element $\sigma\in\omega(X)$ with $\sigma_i=x$ for all $i<l(\sigma)=n$. We write $\omega^x$ at the place of~$\omega^x\cdot 1$. In the following, we use some basic ordinal arithmetic that is readily proved in our setting (cf.~\cite{sommer95}). To show that there are no embeddings into initial segments, we first consider the following special case:
    \begin{claim}
    There is no embedding $f:\omega^x\to\omega^y\cdot n$ for $x>_X y$ and $n\in\mathbb N$.
    \end{claim}
    {\renewcommand{\qedsymbol}{$\diamond$}
    \begin{proof}[Proof of the claim]
    Aiming at a contradiction, we assume that $f$ is an embedding as in the claim. By~the pigeonhole principle, we find $k,m<n$ with
    \begin{equation*}
        \omega^y\cdot m\leq f(\omega^y\cdot k)<f(\omega^y\cdot(k+1))<\omega^y\cdot(m+1).
    \end{equation*}
    This allows us to write
    \begin{equation*}
        f(\omega^y\cdot(k+1))=\omega^y\cdot m+\omega^{y'}\cdot n'+\sigma\quad\text{with $y'<y$ and $\sigma<\omega^{y'}$}.
    \end{equation*}
For future reference, we note that $y'$ and $n'$ can be computed from $y$ and $k$ relative to the given $f$. We now get a map
\begin{equation*}
f':\omega^y\to\omega^{y'}\cdot(n'+1)\quad\text{with}\quad f(\omega^y\cdot k+\tau)=\omega^y\cdot m+f'(\tau).
\end{equation*}
The idea is to iterate the construction to find $y>y'>\ldots$, against the assumption that $X$ is well founded. To perform the iteration over~$\mathsf{RCA_0}$, we do not form the sequence of functions $f,f',\ldots$ but use recursion to compute elements $y_i\in X$ and numbers $k_i,m_i$ that encode the relevant information. In the base of the recursion, we declare that $y_0,k_0,m_0$ coincide with $y,k,m$ from above. For the recursion step, we introduce the abbreviations
\begin{equation*}
\eta_i=\omega^{y_0}\cdot m_0+\ldots+\omega^{y_{i-1}}\cdot m_{i-1}\quad\text{and}\quad\xi_i=\omega^{y_0}\cdot k_0+\ldots+\omega^{y_{i-1}}\cdot k_{i-1}.
\end{equation*}
Let us inductively assume that we have
\begin{equation*}
\eta_i+\omega^{y_i}\cdot m_i\leq f(\xi_i+\omega^{y_i}\cdot k_i)<f(\xi_i+\omega^{y_i}\cdot(k_i+1))<\eta_i+\omega^{y_i}\cdot(m_i+1).
\end{equation*}
Note that we have already established that this holds for~$i=0$. As above, we now find $y_{i+1}<y_i$ and $n''\in\mathbb N$ with
\begin{equation*}
f(\xi_i+\omega^{y_i}\cdot(k_i+1))<\eta_i+\omega^{y_i}\cdot m_i+\omega^{y_{i+1}}\cdot(n''+1).
\end{equation*}
For $\eta_{i+1}=\eta_i+\omega^{y_i}\cdot m_i$ and $\xi_{i+1}=\xi_i+\omega^{y_i}\cdot k_i$, we learn that any $k\in\mathbb N$ validates
\begin{equation*}
\eta_{i+1}\leq f(\xi_{i+1}+\omega^{y_{i+1}}\cdot k)<f(\xi_i+\omega^{y_i}\cdot(k_i+1))<\eta_{i+1}+\omega^{y_{i+1}}\cdot(n''+1).
\end{equation*}
By the pigeonhole principle, a bounded search will thus yield~$k_{i+1},m_{i+1}\leq n''$ with
\begin{align*}
\eta_{i+1}+\omega^{y_{i+1}}\cdot m_{i+1}&\leq f(\xi_{i+1}+\omega^{y_{i+1}}\cdot k_{i+1})<{}\\
{}&<f(\xi_{i+1}+\omega^{y_{i+1}}\cdot(k_{i+1}+1))<\eta_{i+1}+\omega^{y_{i+1}}\cdot(m_{i+1}+1),
\end{align*}
as needed to complete the recursion step.
\end{proof}}
\noindent More generally, we now derive a contradiction from the assumption that $f:I\to I_0$ is an embedding between initial segments $I_0\subsetneq I\subseteq\omega(X)$. Pick a $\sigma\in I\backslash I_0$ and note that $f(\sigma)\in I_0$ entails $f(\sigma)<\sigma$. For $j\leq l(\sigma)$ we write $\sigma[j]=\langle\sigma_0,\ldots,\sigma_{j-1}\rangle$. We use induction on~$j$ to prove $\sigma[j]\leq f(\sigma[j])$. In view of $\sigma[l(\sigma)]=\sigma$, this yields the desired contradiction when we reach~$j=l(\sigma)$. For $j=0$ we note that $\sigma[0]=\langle\rangle$ is the smallest element of~$\omega(X)$. In the induction step, we have $\sigma[j+1]=\sigma[j]+\omega^{\sigma_j}$. If we had $f(\sigma[j+1])<\sigma[j+1]$, we would find $y<\sigma_j$ and $n\in\mathbb N$ with
\begin{equation*}
\sigma[j]\leq f(\sigma[j])<f(\sigma[j]+\omega^{\sigma_j})<\sigma[j]+\omega^y\cdot n.
\end{equation*}
This would yield an embedding
\begin{equation*}
    f':\omega^{\sigma_j}\to\omega^y\cdot n\quad\text{with}\quad f(\sigma[j]+\tau)=\sigma[j]+f'(\tau),
\end{equation*}
against the claim that was proved above.
\end{proof}

The following special case is interesting insofar as $\omega^\omega=\omega(\omega)$ is the proof theoretic ordinal of~$\RCA_0$, so that the latter cannot prove its well foundedness (cf.~\cite{kreuzer-yokoyama}).

\begin{corollary}[$\RCA_0$]\label{cor:omega-omega-rca}
The order $\omega^\omega$ is a weak well order.
\end{corollary}

Our next result will be used in order to lower the base theory in Theorem~\ref{thm:atr-wwo}, which will then supersede it. 

\begin{corollary}[$\RCA_0$]\label{cor:wwo-aca}
    Arithmetic comprehension follows from the statement that every weak well order is a well order.
\end{corollary}
\begin{proof}
    In view of Theorem~\ref{thm:ord_exp}, the given statement entails that $\omega(X)$ is a well order whenever the same holds for~$X$. The latter entails arithmetic comprehension, as proved by J.-Y.~Girard~\cite{girard87} and J.~Hirst~\cite{hirst94}.
\end{proof}

Together with Proposition~\ref{prop:wwo-to-wo}, we obtain the following.

\begin{corollary}[$\RCA_0$]\label{cor:Fraisse-aca}
Fra\"iss\'e's conjecture entails arithmetic comprehension.
\end{corollary}

The previous corollary fills a small gap in Shore's proof that Fra\"iss\'e's conjecture entails arithmetic transfinite recursion, which was mentioned in the introduction. Let us note that our argument uses Fra\"iss\'e's conjecture for arbitrary linear orders, while Shore considers restricted versions of the conjecture for well orders. In the next section, we show how the aforementioned gap can be filled for these versions as well. {We now complete the proof of a main result of this paper.}

 \begin{theorem}[$\RCA_0$]\label{thm:atr-wwo}
    The following are equivalent:
    \begin{enumerate}[label=(\roman*)]
    \item arithmetic transfinite recursion,
    \item every weak well order is a well order.
    \end{enumerate}
\end{theorem}
\begin{proof}
    The forward implication holds by {Corollary~\ref{cor:wwo-to-wo-over-atr}}. For the other implication, Corollary~\ref{cor:wwo-aca} allows us to argue in~$\ACA_0$. Over the latter, arithmetic transfinite recursion follows from \Fraisse's conjecture for indecomposable well orders, by a previously mentioned result of Shore~\cite{Shore}. We can conclude by Proposition~\ref{prop:backward_implication}.
\end{proof}

When we have $X\cong\omega(X)$, Theorem~\ref{thm:ord_exp} tells us that $X$ is a well order precisely if it is a weak well order. We want to draw the same conclusion under the prima facie weaker assumption that $\omega(X)$ embeds into~$X$. This is not a direct consequence of the cited theorem (though we will see that it is a consequence of its proof), because there is no elementary proof that weak well orders are preserved under embeddings (or suborders), as our next observation shows.

\begin{lemma}[$\RCA_0$]\label{lem:wwo-weak-embed}
The following are equivalent:
\begin{enumerate}[label=(\roman*)]
    \item Every weak well order is a well order.
    \item If $X\le_w Y$ and $Y$ is a weak well order, then so is $X$.
\end{enumerate}
\end{lemma}
\begin{proof}
Assuming~(i) and the premise of~(ii), we learn that~$Y$ and hence~$X$ is a well order. By Lemma \ref{lem:wo_to_wwo} it follows that~$X$ is a weak well order. We now assume~(ii) and derive the contrapositive of~(i). Suppose that $X$ is ill founded: any descending sequence witnesses $\omega^*\le_w X$, and $\omega^*\ni i\mapsto i+1$ is an embedding into a proper initial segment. So $\omega^*$ is no weak well order, and by~(ii) the same holds for~$X$.
\end{proof}

Concerning the following result, we note that the well orders with $\omega(X)\leq_w X$ are the $\varepsilon$-numbers. 

\begin{proposition}[$\RCA_0$]\label{prop:wo-wwo-epsilon}
    Consider a linear order~$X$. If we have $\omega(X)\leq_w X$, then $X$ is a well order precisely if it is a weak well order.
\end{proposition}
\begin{proof}
    The forward implication holds by Lemma~\ref{lem:wo_to_wwo}. For the converse direction, we assume that we have $\omega(X)\leq_w X$ and that~$X$ is ill founded. As in the proof of Theorem~\ref{thm:ord_exp}, we learn that $\mathbb Q$ embeds into $\omega(X)$ and hence into~$X$. By Lemma~\ref{lem:wwo_scatt} it follows that $X$ is no weak well order.
\end{proof}

One might have hoped that the approach from the proof of Corollary~\ref{cor:wwo-aca} could be extended. Specifically, H.~Friedman has shown that arithmetic transfinite recursion is equivalent to the statement that $\varphi_X(0)$ is well founded for any well order~$X$, again over~$\RCA_0$ (see~\cite{marcone-montalban-veblen,rathjen-weiermann-atr} for published proofs). Here $\varphi_X(0)$ is a notation system related to the Veblen hierarchy. In view of Proposition~\ref{prop:backward_implication}, it would seem conceivable that $\RCA_0$ proves~$\varphi_X(0)$ to be a weak well order for any well order~$X$. {Before Corollary~\ref{cor:wwo-to-wo-over-atr} had been established,} one might even have tried to give a proof that~$\Gamma_0=\min\{\alpha\,|\,\varphi_\alpha(0)=\alpha\}$ is a weak well order, perhaps in~$\RCA_0$ but at least in~$\ATR_0$, which has proof-theoretic ordinal~$\Gamma_0$. Parallel to Corollary~\ref{cor:Fraisse-aca}, this would have lead to the spectacular result that $\ATR_0$ does not prove Fra\"iss\'e's conjecture. However, the following result shows that none of the indicated possibilities can materialize. This yields an interesting contrast with Corollary~\ref{cor:omega-omega-rca}.

\begin{corollary}\label{cor:unprov-wwo}
The following holds with respect to the standard notation systems for proof-theoretic ordinals (see, e.\,g.,~\cite{pohlers-proof-theory}):
\begin{enumerate}[label=(\alph*)]
    \item In $\ACA_0$ one cannot prove that $\varepsilon_0=\varphi_1(0)$ is a weak well order.
    \item In $\ATR_0$ one cannot prove that $\Gamma_0$ is a weak well order.
\end{enumerate}
\end{corollary}
\begin{proof}
The point is that embeddings $\omega(\varepsilon_0)\leq_w\varepsilon_0$ and $\omega(\Gamma_0)\leq_w\Gamma_0$ are implicit in the standard notation systems. Hence by Proposition~\ref{prop:wo-wwo-epsilon}, the result reduces to the claim that $\ACA_0$ and $\ATR_0$ cannot prove the well foundedness of~$\varepsilon_0$ and $\Gamma_0$, respectively. This is true because the latter are the proof-theoretic ordinals of the indicated theories (see again \cite{pohlers-proof-theory}). Let us point out that we could have invoked Corollary~\ref{cor:wwo-to-wo-over-atr} rather than Proposition~\ref{prop:wo-wwo-epsilon} in order to prove~(b).
\end{proof}

\section{Fra\"iss\'e's conjecture and $\Sigma^0_2$-induction}\label{sect:Fraisse-Sigma2}

In the present section, we show that Fra\"iss\'e's conjecture for well orders entails $\Sigma^0_2$-induction over~$\RCA_0$. More precisely, it will suffice to assume either of two consequences of Fra\"iss\'e's conjecture, which assert that the countable well orders contain no infinitely descending sequences and no infinite antichains, respectively. As noted in the introduction, this fills a small gap in Shore's~\cite{Shore} proof that Fra\"iss\'e's conjecture implies arithmetic transfinite recursion over~$\RCA_0$. The issue with this proof is that it uses the well foundedness of $\omega^\omega$, which $\RCA_0$ cannot prove.

In the case of Fra\"iss\'e's conjecture for arbitrary linear orders (not necessarily~well founded), the aforementioned gap is filled by our Corollary~\ref{cor:omega-omega-rca} in conjunction with Proposition~\ref{prop:wwo-to-wo} (or by the stronger Corollary~\ref{cor:Fraisse-aca}). To accommodate the~restriction of Fra\"iss\'e's conjecture to well orders, we give an argument that is similar to Shore's but works with smaller ordinals. This will necessarily involve some new idea (which we explain after the proof), because Shore uses an infinite supply of indecomposable well orders, which cannot be bounded below~$\omega^\omega$. We first consider Fra\"iss\'e's conjecture for descending sequences of well orders, which Shore denotes ($\mathcal{WF}1$).

\begin{theorem}[$\RCA_0$]\label{thm:wf1}
The principle of $\Sigma^0_2$-induction is implied by the following restriction of Fra\"iss\'e's conjecture: for any infinite sequence of well orders~$L_0,L_1,\ldots$ such that each $L_{n+1}$ embeds into~$L_n$, there are $i<j$ such that $L_i$ embeds into~$L_j$.
\end{theorem}

Let us note that we can get $j+1=i$ when we know that embeddability is transitive along finite chains of arbitrary length. However, the obvious proofs of this fact use a substantial amount of induction or choice. Alternatively, we could require the stronger condition that $L_k$ embeds into~$L_m$ for all~$m<k$, which is satisfied in the following construction.

\begin{proof}
Consider a $\Sigma^0_2$-formula $\psi(x)\equiv\exists u\forall v\,\phi(x,u,v)$. For arbitrary~$n\in\mathbb N$, we will construct well orders $N_0,N_1,\ldots$ such that $N_k$ embeds into~$N_m$ for all $m<k$ and any embedding $F:N_i\to N_j$ with $i<j$ allows us to compute a set $X\subseteq\mathbb N$ with
\begin{equation*}
\forall x<n\big(\psi(x)\leftrightarrow x\in X\lor\exists u\leq i\forall v\,\phi(x,u,v)\big).
\end{equation*}
Here $\exists u\leq i\forall v\,\phi(x,u,v)$ is equivalent to $\forall w\exists u\leq i\forall v\leq w\,\phi(x,u,v)$, by the principle of strong $\Sigma^0_1$-bounding (see Exercise~II.3.14 of~\cite{simpson09}). So induction for $\psi(x)$ up to~$n$ is~reduced to an instance of $\Pi^0_1$-induction, which is available in~$\mathsf{RCA_0}$.

We would like to have $N_i=\sum_{y<2n}1+M_{i,y}$ with $M_{i,2x+1}=\omega$ and
\begin{equation*}
M_{i,2x}\cong\begin{cases}
\max(0,u-i) & \text{if $\psi(x)$ holds and $u$ is minimal with $\forall v\,\phi(x,u,v)$},\\
\omega & \text{if $\neg\psi(x)$ holds}.
\end{cases}
\end{equation*}
However, this characterization of~$M_{i,2x}$ cannot serve as our definition, because the case distinction is undecidable. In order to resolve this issue, we first define a computable function~$(x,u)\mapsto k_x(u)$, which may be partial. When $k_x(u-1)$ is defined (where we read $k_x(-1)=0$ for the base case), we let $k_x(u)$ be the minimal~$k>k_x(u-1)$ such that there is a~$v\leq k$ with~$\neg\phi(x,u,v)$, if such a~$v$ can be found. If there is no such~$v$ or if~$k_x(u-1)$ is undefined, then $k_x(u)$ is undefined. Note that $k_x(u)$ is undefined precisely if there is a~$u'\leq u$ with $\forall v\,\phi(x,u',v)$. While the latter is undecidable as a property of~$u$, we can decide whether a given number has the form~$k_x(u)$, since we have $k_x(0)<k_x(1)<\ldots$ and hence $u\leq k_x(u)$. This allows us to form
\begin{equation*}
M_{i,2x}=\{k\in\mathbb N\,|\,\text{we have $k=k_x(u)$ for some $u\geq i$}\},
\end{equation*}
which we consider as a suborder of~$\mathbb N$. To confirm the characterization from above, we first assume that $u$ is minimal with $\forall v\,\phi(x,u,v)$. As noted above, this means that $k_x(u')$ is defined precisely for~$u'<u$, which clearly yields $M_{i,2x}\cong\min(0,u-i)$. Now assume that we have $\neg\psi(x)$. Then $k_x(u)$ is defined for all~$u$, so that we indeed obtain~$M_{i,2x}\cong\omega$. In particular, it follows that $N_i=\sum_{y<2n}1+M_{i,y}$ is a well order. When we have $m<k$, we clearly get $M_{k,2x}\subseteq M_{m,2x}$ for all~$x<n$, which entails that~$N_k$ embeds into~$N_m$. Thus the given consequence of Fra\"iss\'e's conjecture yields an embedding $F:N_i\to N_j$ for some~$i<j$.

To simplify the notation for elements of~$N_l$, we write $1+M_{l,2x}=\{0\}\cup M_{l,2x}$ and identify $1+M_{l,2x+1}=1+\omega$ with~$\omega$. Each $k\in 1+M_{l,y}$ yields an element~$(y,k)$ in the $y$-th summand of $N_l=\sum_{y<2n}1+M_{l,y}$. One can establish $F(y,0)\geq(y,0)$ by induction on $y<2n$, using that $M_{j,y}$ embeds into~$M_{i,y}$ and that no well order embeds into a proper initial segment of itself. A crucial feature of our construction is that we also get~$F(2x+1,k)<(2x+2,0)$ for all~$k\in\mathbb N$, which means that $F$ maps $M_{i,2x+1}\cong\omega$ into~$M_{j,2x+1}$. Intuitively, this is true because $N_i$ and $N_j$ have the same number of summands~$\omega$. Formally, we argue by induction from $x=n-1$ down to~$x=0$, where we interpret $(2n,0)$ as an additional point above~$N_j$, so that the claim for~$x=n-1$ is immediate. For the induction step, we derive a contradiction from the assumption that we have
\begin{equation*}
(2x,0)\leq F(2x-1,k)\quad\text{and}\quad F(2x+1,0)<(2x+2,0).
\end{equation*}
These inequalities entail that $F$ induces an embedding of $\omega+1+M_{i,2x}$ into a proper initial segment of $1+M_{j,2x}+\omega$. It follows that $M_{j,2x}$ must infinite, which can only hold if we have $\neg\psi(x)$ and hence $M_{i,2x}\cong\omega\cong M_{j,2x}$. But then we have an embedding of $\omega+\omega$ into a proper initial segment of itself, which is impossible.

Let us note that $F$ can map elements of $M_{i,2x}$ into $M_{j,2x+1}\cong\omega$ rather than~$M_{j,2x}$. To control this phenomenon, we form the set
\begin{equation*}
X=\{x<n\,|\,\text{there is a $k\in M_{i,2x}$ with $F(2x,k)\geq (2x+1,0)$}\},
\end{equation*}
which relies on bounded $\Sigma^0_1$-comprehension in $\RCA_0$ (see Theorem~II.3.9 of~\cite{simpson09}). If we have $x\in X$, there is a nonempty final segment $S\subseteq M_{i,2x}$ such that $F$ induces an embedding of $S+M_{i,2x+1}$ into~$M_{j,2x+1}\cong\omega$ (recall $F(2x+1,k)<(2x+2,0)$ from above). But then $M_{i,2x}$ cannot be isomorphic to~$\omega$, which means that we must have~$\psi(x)$. To confirm the equivalence from the beginning of this proof, we now assume that we have~$\psi(x)$ but~$x\notin X$. We may then consider the minimal~$u$ with $\forall v\,\phi(x,u,v)$. From $x\notin X$ we can infer that $F$ induces an embedding of $M_{i,2x}$ into~$M_{j,2x}$. We thus obtain
\begin{equation*}
M_{i,2x}\cong\min(0,u-i)\leq\min(0,u-j)\cong M_{j,2x}.
\end{equation*}
Given that we have $i<j$, this can only be true if we have $u\leq i$. In other words, we can conclude $\exists u\leq i\forall v\,\phi(x,u,v)$, as in the desired equivalence.
\end{proof} 

The original argument by Shore uses different indecomposable ordinals at the place of the summands $1+M_{i,2x}+\omega$ from the previous proof. Our main new idea is that one can use copies of $\omega$ as separators between the summands $1+M_{i,2x}$ if one employs a set~$X$ to recover information that is lost when a summand maps into a separator. We now consider Fra\"iss\'e's conjecture for antichains of well orders, which Shore denotes~($\mathcal{WF}2$). Our modifications have the nice side effect that the proofs for ($\mathcal{WF}1$) and ($\mathcal{WF}2$) become more similar than in Shore's original paper.

\begin{theorem}[$\RCA_0$]\label{thm:wf2}
The principle of $\Sigma^0_2$-induction is implied by the following restriction of Fra\"iss\'e's conjecture: for any infinite sequence of well orders~$L_0,L_1,\ldots$, there are $i\neq j$ such that $L_i$ embeds into~$L_j$.
\end{theorem}
\begin{proof}
Fix a $\Sigma^0_2$-formula $\psi(x)\equiv\exists u\forall v\,\phi(x,u,v)$ and some~$n\in\mathbb N$. As in the previous proof, we find well orders $N'_i=\sum_{y<4n}1+M'_{i,y}$ with $M'_{i,4x+1}=M'_{i,4x+3}=\omega$ and
\begin{align*}
M'_{i,4x}&\cong\begin{cases}
\max(0,u-i) & \text{if $\psi(x)$ holds and $u$ is minimal with $\forall v\,\phi(x,u,v)$},\\
\omega & \text{if $\neg\psi(x)$ holds},
\end{cases}\\
M'_{i,4x+2}&\cong\begin{cases}
i+u & \text{if $\psi(x)$ holds and $u$ is minimal with $\forall v\,\phi(x,u,v)$},\\
\omega & \text{if $\neg\psi(x)$ holds}.
\end{cases}
\end{align*}
We obtain an embedding $F':N'_i\to N'_j$ for some indices~$i\neq j$.

Once again, we write $1+M'_{l,2y}=\{0\}\cup M'_{l,2y}$ and identify $1+M'_{l,2y+1}=1+\omega$ with $\omega$. For each $y<2n$, the orders $M'_{i,2y}+\omega$ and $M'_{j,2y}+\omega$ are isomorphic. We thus get $F'(2y,0)\geq(2y,0)$ by induction on~$y$. In contrast to the previous proof, the inequality $F'(2y+1,0)\geq(2y+1,0)$ is only available for every second~$y$, where the parity of the admissible~$y$ corresponds to the order between~$i$ and~$j$. Nevertheless, we again get $F'(2y+1,k)<(2y+2,0)$ for any $y<2n$ and all~$k\in\mathbb N$.

In case we have~$i<j$, we can conclude as before. So now assume~$i>j$. We put
\begin{equation*}
X':=\{x<n\,|\,\text{there is a $k\in M'_{i,4x+2}$ with $F'(4x+2,k)\geq(4x+3,0)$}\}.
\end{equation*}
It is still true that $x\in X'$ entails $\psi(x)$. To complete the proof, we show that the converse implication holds for any~$x<n$. Aiming at a contradiction, we assume that we have $\psi(x)$ but also~$x\notin X'$. The latter entails that $F'$ induces an embedding of $M'_{i,4x+2}$ into~$M'_{j,4x+2}$. For the minimal $u$ with $\forall v\,\phi(x,u,v)$, we must thus have
\begin{equation*}
M'_{i,4x+2}\cong i+u\leq j+u\cong M'_{j,4x+2}.
\end{equation*}
This, however, contradicts the assumption that we have~$i>j$.
\end{proof}

We now reaffirm that the following result of Shore~\cite{Shore} holds with the indicated base theory. The result remains valid when Fra\"iss\'e's conjecture for well orders~is restricted to either of the principles ($\mathcal{WF}1$) and ($\mathcal{WF}2$) mentioned above. To~see this for ($\mathcal{WF}1$), one uses our Theorem~\ref{thm:wf1} and Shore's Theorem~1.2 to reach arithmetic comprehension. Given the latter, one can conclude by Shore's proof of his Theorem~3.7 (see~\cite{Shore} for all cited results by Shore). For ($\mathcal{WF}2$) one invokes the proof of Shore's Theorem~3.8, where his Theorem~3.1(ii) is restored by our Theorem~\ref{thm:wf2}. 

\begin{theorem}[Shore~\cite{Shore}]
Over the theory~$\RCA_0$, arithmetic transfinite recursion is equivalent to Fra\"iss\'e's conjecture for well orders, i.\,e.~to the statement that any infinite sequence of well orders  $L_0,L_1,\ldots$ admits $i<j$ such that $L_i$ embeds into~$L_j$.
\end{theorem}

We conclude with a discussion of `natural' descriptions of orders.

\begin{remark}\label{rmk:natural}
Given an axiom system of arbitrary strength, one can produce a recursive index of an order isomorphic to~$\omega$ such that the axiom system cannot prove that the index describes a well order, as noted by G.~Kreisel (see e.\,g.~\cite{rathjen-realm}). In proof theory, this has lead to a discussion about `natural' or `canonical' descriptions of well orders. While one may not expect a definitive explication of `natural', it is possible to isolate relevant features and to argue that specific descriptions like the standard notation system for~$\varepsilon_0$ are natural (cf.~\cite{feferman-ordinal-notations} or the more recent~\cite{beklemishev-natural-ordinals}). For finite orders, it should not be too controversial to assert that the positive integers provide canonical representatives. In the proof of Theorem~\ref{thm:wf1}, we did not use these representatives to define the orders~$M_{i,2x}$. It was indeed crucial to work with `nonstandard' descriptions, for which we could not decide whether the resulting orders were finite. Similar phenomena occur in other parts of Shore's proof that Fra\"iss\'e's conjecture entails arithmetic transfinite recursion. In view of this observation, one may wonder how much strength Fra\"iss\'e's conjecture retains when we only admit `natural' descriptions of orders. It is not clear whether this question can be answered or even formulated in a fully satisfying way. At the same time, the positive and negative results of the previous section seem to yield some relevant insights. On the positive side, the proof of the forward direction in Theorem~\ref{thm:ord_exp} is based on natural properties of ordinal exponentiation. As this direction was the crucial step towards \mbox{Corollary~\ref{cor:Fraisse-aca}}, it seems justified to conclude that arithmetic comprehension follows from Fra\"iss\'e's conjecture for `natural' linear orders. On the negative side, the given line of argument cannot be extended beyond arithmetic comprehension, as shown by Corollary~\ref{cor:unprov-wwo}. The crucial property behind this corollary is closure under exponentiation (cf.~Proposition~\ref{prop:wo-wwo-epsilon}), which can be seen as a minimal condition on natural notations for larger proof-theoretic ordinals.
\end{remark}

\bibliographystyle{plain}
\bibliography{references}
\end{document}